\documentclass{article}
\usepackage{graphicx}
\usepackage{xcolor}
\usepackage{amsmath,amsthm,amssymb}
\usepackage{comment}
\usepackage{booktabs}
\usepackage{etoolbox}
\usepackage{geometry}
\usepackage{authblk}
\title{\Large Arithmetic of elliptic curves induced by regular Diophantine triples}
\author{Nikola Ad\v zaga}
\makeatletter

\patchcmd{\@setauthors}{\scshape}{\relax}{}{}
\patchcmd{\@author}{\MakeUppercase}{\relax}{}{}
\renewcommand{\author}[1]{\def\@author{#1}}

\makeatother
\AtEndDocument{%
  \par\bigskip
  \noindent\textsc{Department of Mathematics, University of Zagreb Faculty of Civil Engineering}\par
  \noindent\emph{Email address:} \texttt{nikola.adzaga@grad.unizg.hr}\par
}

\date{30th July 2026}

\usepackage{xurl}
\usepackage{hyperref}

\usepackage{cleveref}
\usepackage[style=alphabetic]{biblatex}
\AtBeginBibliography{\sloppy}

\theoremstyle{plain}
\newtheorem{theorem}{Theorem}
\newtheorem{lemma}[theorem]{Lemma}
\newtheorem{proposition}[theorem]{Proposition}
\newtheorem{corollary}[theorem]{Corollary}

\newtheorem{remark}[theorem]{Remark}

\newtheorem{setup}[theorem]{Assumptions}

\crefname{theorem}{Theorem}{Theorems}
\crefname{lemma}{Lemma}{Lemmata}
\crefname{corollary}{Corollary}{Corollaries}
\crefname{proposition}{Proposition}{Propositions}
\crefname{definition}{Definition}{Definitions}
\crefname{definition-proposition}{Definition-Proposition}{Definition-Propositions}
\crefname{conjecture}{Conjecture}{Conjectures}
\crefname{question}{Question}{Questions}
\crefname{example}{Example}{Examples}
\crefname{algorithm}{Algorithm}{Algorithms}
\crefname{remark}{Remark}{Remarks}
\crefname{setup}{Assumptions}{Assumptions}

\newcommand{\Z}{\mathbb{Z}}
\newcommand{\Q}{\mathbb{Q}}
\newcommand{\F}{\mathbb{F}}
\newcommand{\twotimessix}{\Z/2\Z \times \Z/6\Z}

\newcommand{\tors}{\mathrm{tors}}

\begin{document}

\maketitle
\begingroup
\renewcommand{\thefootnote}{}
\footnotetext{%

}
\addtocounter{footnote}{-1}
\endgroup

\begin{abstract}
    We study elliptic curves induced by regular Diophantine triples, with emphasis on their torsion subgroups. We show that an elliptic curve \(E\) induced by a regular Diophantine triple in integers necessarily has torsion subgroup \(E(\Q)_\tors \cong \Z/2\Z \times \Z/2\Z\). Moreover, we develop a criterion for when such an elliptic curve acquires a point of order \(3\) over a quadratic field. For a particular family \(\{ k-1, k+1, 4k\}\), we use it to show that this does not happen. Finally, we study both the torsion and the generic rank of a family of elliptic curves induced by the \(D(-k^2)\)-triple \(\{1, 2k^2, 2k^2+2k+1\}\).\\
\end{abstract}

\section{Introduction and the main result}

Diophantus of Alexandria studied sets of rational numbers with the property that the product of any two distinct elements, increased by \(1\), is a perfect square. 
A natural question is how large such sets can be. 
In the case of positive integers, the subject has a long history. 
Fermat found the Diophantine quadruple in positive integers, namely \(\{1,3,8,120\}\), 
while Baker and Davenport proved that if \(\{1,3,8,d\}\) is a Diophantine quadruple, then \(d=120\) \cite{BaDa}. 
In other words, the triple \(\{1,3,8\}\) has a unique extension to a quadruple. 
This result motivated a lot of research on extension problems for Diophantine triples, and more generally on the structure and size of Diophantine \(m\)-tuples. One of the highlights of this decades-long research is a proof by He, Togb\'e and Ziegler that there does not exist a Diophantine quintuple in positive integers \cite{HeToZi}. For more on background and previous research, we recommend \cite{DujeBook,DujeWeb}.

Our approach is to place these problems in a broader arithmetic-geometric setting, with the presentation of the methods being as expository and introductory as possible. A frequent theme will be that studying the behaviour over a larger structure can be used to deduce results in the original structure: for example, behaviour over \(\Q\) is easily used to get integral conclusions, while sometimes the behaviour over quadratic fields is used to get conclusions over \(\Q\). It is also our hope to motivate researchers in the field of Diophantine \(m\)-tuples to apply these methods.

Let \(\{a,b,c\}\) be a Diophantine triple in positive integers. 
To extend it to a quadruple, one has to solve the system
\begin{equation}\label{eq:system}
ax+1=\square,\quad bx+1=\square,\quad cx+1=\square.
\end{equation}
It is therefore natural to attach to \(\{a,b,c\}\) the elliptic curve
\[
E:\ y^2=(ax+1)(bx+1)(cx+1),
\]
the \emph{induced} elliptic curve.
Every integer \(d\) for which \(\{a,b,c,d\}\) is a Diophantine quadruple gives an integral point on \(E\) (with \(x=d\)).
In this way, questions about extendibility of Diophantine triples become closely connected with  properties of the induced elliptic curves.

This point of view was developed by Dujella and others. 
In particular, in  \cite{DujeJTNB,DujeiMiljen}, it was proved that if \(a,b,c\) are positive integers, then the torsion subgroup of the induced curve \(E(\Q)_\tors\) can be only
\[
\Z/2\Z\times\Z/2\Z
\quad\text{or}\quad
\Z/2\Z\times\Z/6\Z.
\]
An analogous statement for \(D(4)\)-triples (where one changes \(1\) to \(4\), or more generally to \(n\), in the system \eqref{eq:system}) was proved by Dujella and Miki\'c \cite{DujeiMiljen}, with the single exceptional triple \(\{-1,3,4\}\), whose induced curve has torsion group \(\Z/2\Z\times\Z/4\Z\). 
On the other hand, Ono and Kwon \cite{Ono,Kwon1997} gave a parametrization of the case
\(\Z/2\Z\times\Z/6\Z\) for curves of the form
\[
y^2=x(x+M)(x+N),
\]
and this parametrization has become a basic tool in excluding the presence of \(6\)-torsion in concrete families.

In this context, several important parametric families have already been studied. For the regular family \(\{k-1,k+1,4k\}\), Dujella \cite{DujeParametric} proved that the induced curve has torsion group \(\Z/2\Z\times\Z/2\Z\). 
For the family \(\{1,3,c\}\), Dujella and Pethő \cite{DujePetho} showed that the induced curve again has torsion group \(\Z/2\Z\times\Z/2\Z \). This result was then significantly extended by Mikić in \cite{MikicRMJ2015}. The same holds for the Fibonacci family \(\{F_{2k},F_{2k+2},F_{2k+4}\}\) \cite{DujeJTNB}.

Together with the fact that, so far, no example of an elliptic curve induced by a triple of positive integers with torsion $\mathbb{Z}/2\mathbb{Z}\times\mathbb{Z}/6\mathbb{Z}$ has been found, these results suggest that the possibility of torsion $\mathbb{Z}/2\mathbb{Z}\times\mathbb{Z}/6\mathbb{Z}$ is very restricted.

In this paper we consider the regular case. 
Recall that a Diophantine triple \(\{a,b,c\}\) is called \emph{regular} if
\[
c=a+b+2r,\quad r^2=ab+1.
\]
Regular triples are a natural first choice one tries to understand.

Our main result is the following.

\begin{theorem}
For every regular integral Diophantine triple \(\{a,b,c\}\), the induced elliptic curve
\[
E:\ y^2=(ax+1)(bx+1)(cx+1)
\]
has torsion subgroup
\( E(\Q)_{\tors}\cong \Z/2\Z\times \Z/2\Z \).
\end{theorem}


We then move on to study torsion of the induced elliptic curves over quadratic fields, giving a criterion (Proposition \ref{prop:lambda-parametrization}) for when such an elliptic curve acquires a point of order \(3\). We apply this to elliptic curves induced by \( \{k-1, k+1, 4k\} \), for integer \(k \).

Finally, we study both the torsion and the generic rank of an elliptic curve induced by a \(D(-k^2)\)-triple \( \{1, 2k^2, 2k^2+2k+1 \} \), finding further support for a conjecture posed in \cite{AF} on possible extensions of such triples. We prove that the generic rank of this induced curve is \(0\).

While the first part dealing with \(\Q\)-torsion of elliptic curves induced by regular triples in integers can be seen as elementary, for the later parts we had to use some code in Magma, which is available at
\url{https://github.com/NikolaAdzaga/TorsionRegularTriples}.


\section{Setup and divisibility conclusions}

Without loss of generality, we can assume \(a < b < c\), and moreover, by the following, we can even restrict our proofs to positive integers.

\begin{remark}
    Since an integral Diophantine triple has one sign only, we can change that sign. The two curves $E: y^2=(ax+1)(bx+1)(cx+1)$ and $E': y^2=(x+ab)(x+bc)(x+ca)$ are isomorphic over $\Q$, so we use them interchangeably. Hence, changing the signs of $a,b,c$ (and $x$ if we use $E$ rather than $E'$) shows we can focus on triples in positive integers.
\end{remark}

We start from the parametrization of the \( \Z/2\Z\times\Z/6\Z \)-torsion.

\begin{theorem}[{\protect{\cite[Main Theorem 1]{Ono}, \cite{Kwon1997}}}]\label{thm:OnoCorrected}
Let \( E(M,N): y^2=x(x+M)(x+N)\), with \(M,N\in \Z\) and \(MN(M-N)\neq 0\).
If \( E(\Q)_{\tors}\cong \Z/2\Z\times \Z/6\Z\), then there exist integers \(d\neq 0\), \(\alpha,\beta\in \Z\) with
\(\gcd(\alpha,\beta)=1\) such that
\( M=d^2\,\alpha^3(\alpha+2\beta),\, N=d^2\,\beta^3(\beta+2\alpha) \).
\end{theorem}

\begin{remark}
In Ono's proof of the \( \Z/2\Z\times \Z/6\Z \) case, the line
\(x=a^2b^2\) should be \(x=d'a^2b^2\) for some integer \(d'\), but substituting into the equation of $E$ shows it has to be a square $d'=d^2$. Coprimality of \( \alpha, \beta\)  (these are Ono's \(a, b\) in our notation) can be ensured since, if prime $p|\alpha, \beta$, then $p^4|M, N$, and hence $p^2$ can be pulled into $d$.
Since only \(d^2\) appears, we can assume \(d > 0\).

\end{remark}

\begin{lemma}\label{lem:bezRegularnosti}
Let \(a,b,c\in \Z_{>0}\) with \(a<b<c\), and let \(E: y^2=(ax+1)(bx+1)(cx+1)\).
Assume \(E(\Q)_{\tors}\cong \Z/2\Z\times \Z/6\Z\), and set
\(M=a(c-b)\), \(N=b(c-a)\). Then
\[
c(b-a)=N-M=d^2(\beta-\alpha)(\alpha+\beta)^3.
\]
After replacing \((\alpha,\beta)\) by \((-\alpha,-\beta)\) if necessary, we may assume \(\beta>\alpha > 0\).
\end{lemma}


\begin{proof}
Apply \cref{thm:OnoCorrected} to a model obtained from $y^2=(x+ab)(x+bc)(x+ca)$ after substituting $x \mapsto x-ab$, i.~e.~\(y^2=x(x+M)(x+N)\).
The identity \(N-M=b(c-a)-a(c-b)=c(b-a)\) is immediate, and
\(d^2\left(\beta^3(\beta+2\alpha)-\alpha^3(\alpha+2\beta)\right) = d^2(\beta-\alpha)(\alpha+\beta)^3\).

Since \(a<b<c\) and \(c>0\), we have that \(M,N\), and \(N-M\) are positive, so
\((\beta-\alpha)(\alpha+\beta)^3>0\). Replacing \((\alpha,\beta)\) by \((-\alpha,-\beta)\) if needed, we may assume \(\alpha+\beta>0\), and then \(\beta-\alpha>0\). Since \(M>0\), we conclude \(\beta>\alpha > 0\).
\end{proof}

\begin{setup}\label[setup]{set:regular2x6tors}
Assume that \(\{a,b,c\}\) is a regular integral Diophantine triple in positive integers $a<b<c$, so \(ab+1=r^2\) and \(c=a+b+2r\), where \(r > 0\). Assume that the induced curve \(E: y^2=(ax+1)(bx+1)(cx+1)\) satisfies
\( E(\Q)_{\tors}\cong \Z/2\Z\times \Z/6\Z \).
\end{setup}


 In his PhD thesis, Mikić has shown that, if an induced curve has \(\Z/2\Z\times\Z/6\Z\) torsion, then $a$ and $b$ are even \cite[Teorem 2.19]{Miljen}, and the rank of such a curve is positive \cite[Lema 2.20]{Miljen}. Originally we had shown that \(\gcd(c, b-a) = 2\) (by first showing that it divides \(4\) and analyzing further). This led us to the proof we present in the next section, but in the end, we do not use those statements directly, so we decided to omit those unnecessary elementary proofs.

\section{Change of parameters and the contradiction}

The following lemma is crucial for our proof. 
\begin{lemma}[Change of Parameters Lemma]\label{lem:surfMap}
Let $u = c/2$ and $v = (b-a)/2$. 
Then $M$ and $N$ satisfy:
\[ N-M = 4uv \]
\[ M+N = \frac{3u^4 + 6u^2v^2 - v^4 - 2u^2 + 2v^2 - 1}{2u^2} \]
Furthermore, let $S = \beta+\alpha$ and $D = \beta-\alpha$. Then
\[ N-M = d^2 D S^3 \]
\[ M+N = d^2 \frac{3S^4 + 6S^2D^2 - D^4}{8} \]
\end{lemma}

\begin{remark}
    Before the proof, notice the similar shape of the two statements for the sum: expressions in the numerators are similar, but with different variables.
\end{remark}

\begin{proof}
By regularity, $a = u-v-r$ and $b = u+v-r$. Substituting into $ab+1=r^2$ yields $2ur = u^2 - v^2 + 1$. 
Mapping $M = a(a+2r) = (u-v)^2 - r^2$ and $N = b(b+2r) = (u+v)^2 - r^2$ immediately gives $N-M = 4uv$.
The sum is $M+N = 2u^2 + 2v^2 - 2r^2$, and then substituting $2r^2 = \frac{(u^2-v^2+1)^2}{2u^2}$ gives the first statement for $M+N$.
Now substitute $\alpha = \frac{S-D}{2}$ and $\beta = \frac{S+D}{2}$ into $M = d^2 \alpha^3(\alpha+2\beta)$ and $N = d^2 \beta^3(\beta+2\alpha)$ to immediately get the last statements about sum and difference.
\end{proof}

\begin{remark}
    Since $\beta > \alpha > 0$, $\frac SD = \frac{\beta+\alpha}{\beta-\alpha}>1$.
\end{remark}

\begin{lemma}[The Ratio Inequality]\label{lem:ratioIneq}
Let $F(t) = \frac{3t^4 + 6t^2 - 1}{8t^3}$. Then $F(u/v) > F(S/D)$, which implies $u/v > S/D$.
\end{lemma}

\begin{proof}
Dividing $M+N$ by $N-M$ in both parameter pairs from Lemma \ref{lem:surfMap} gives
\[ F(u/v) - \frac{2u^2 - 2v^2 + 1}{8u^3v} = F(S/D) \]
By the definition of a regular triple, $c = a+b+2r$. Since $a, r > 0$, we have $c > b > b-a$. Thus, $c/2 > (b-a)/2$, meaning $u > v > 0$ holds unconditionally for all regular triples.
Since $u > v > 0$, $2u^2 - 2v^2 + 1 \geqslant 1 > 0$, making the error term strictly positive.
Thus $F(u/v) > F(S/D)$. The derivative $F'(t) = \frac{3(t^2-1)^2}{8t^4}$ is positive for $t > 1$, making $F(t)$ strictly increasing. Therefore, $u/v > S/D$.
\end{proof}

We now want to write \(u/v\) in the similar shape as \(S/D=(\beta+\alpha)/(\beta-\alpha).\)

\begin{lemma}[A structural inequality]\label{lem:BA-lt-betaalpha}
Assume \cref{set:regular2x6tors} hold, and keep the notation
\[
u=\frac c2,\quad v=\frac{b-a}{2},\quad
S=\beta+\alpha,\quad D=\beta-\alpha.
\]
Then
\( \frac{b+r}{a+r}<\frac{\beta}{\alpha} \).
In particular,
\(
r^2>d^2\alpha^2\beta^2,
\)
hence
\(
r > d \alpha\beta.
\)
\end{lemma}

\begin{proof}
By Lemma \ref{lem:ratioIneq}, we have \(\frac uv>\frac SD\). Now
\[
\frac{u}{v}
=\frac{(u+v)+(u-v)}{(u+v)-(u-v)}
=\frac{(b+r)+(a+r)}{(b+r)-(a+r)},
\]
while
\(\displaystyle
\frac{S}{D}
=\frac{\beta+\alpha}{\beta-\alpha}
=\frac{\frac{\beta}{\alpha}+1}{\frac{\beta}{\alpha}-1} \).

Consider the function \(g(z)=\frac{z+1}{z-1}\), for \(z>1\). Since
\(g'(z)=\frac{-2}{(z-1)^2}<0\), the function \(g\) is strictly decreasing on
\((1,\infty)\). As
\[
\frac uv=g\!\left(\frac{b+r}{a+r}\right)
\quad\text{and}\quad
\frac SD=g\!\left(\frac{\beta}{\alpha}\right),
\]
the inequality \(\frac uv>\frac SD\) implies
\(\frac{b+r}{a+r}<\frac{\beta}{\alpha}\), i.~e. \(\alpha(b+r)<\beta(a+r)\).

Now
\[
(a+r)^2=a(a+2r)+r^2=M+r^2,
\quad
(b+r)^2=b(b+2r)+r^2=N+r^2.
\]
Squaring \(\alpha(b+r)<\beta(a+r)\) gives
\(\alpha^2(N+r^2)<\beta^2(M+r^2)\). Substituting
\[
M=d^2\alpha^3(\alpha+2\beta),
\quad
N=d^2\beta^3(\beta+2\alpha),
\]
we obtain
\[
\alpha^2 d^2\beta^3(\beta+2\alpha)+\alpha^2r^2
<
\beta^2 d^2\alpha^3(\alpha+2\beta)+\beta^2r^2.
\]
Rearranging gives
\[
d^2\alpha^2\beta^2(\beta^2-\alpha^2)
<
(\beta^2-\alpha^2)r^2.
\]
Since \(\beta>\alpha>0\), we may divide by \(\beta^2-\alpha^2\) and get
\(r^2>d^2\alpha^2\beta^2\). As all quantities are positive, this implies
\(r>d\alpha\beta\).
\end{proof}

\begin{theorem}\label{thm:no-regular-2x6}
There is no regular integral Diophantine triple \(\{a,b,c\}\) such that the induced elliptic curve
\[
E:y^2=(ax+1)(bx+1)(cx+1)
\]
has torsion subgroup
\[
E(\Q)_{\tors}\cong \Z/2\Z\times \Z/6\Z.
\]
\end{theorem}

\begin{proof}
Assume the contrary. By Lemma \ref{lem:bezRegularnosti}, after replacing \((\alpha,\beta)\) by \((-\alpha,-\beta)\) if necessary, we may assume
\(
\beta>\alpha>0
\).
Also, since only \(d^2\) appears, we may take \(d>0\).

By \cref{lem:BA-lt-betaalpha}, we have
\(
r>d\alpha\beta\). We claim that
\[
a<d\alpha^2
\quad\text{and}\quad
b<d\beta^2.
\]
Suppose first that \(a\geqslant d\alpha^2\). Then
\(
a+r\geqslant d\alpha^2+r\), hence
\[
(a+r)^2\geqslant (d\alpha^2+r)^2
= r^2+2d\alpha^2r+d^2\alpha^4.
\]
Since \(r>d\alpha\beta\), we get
\(
2d\alpha^2r>2d^2\alpha^3\beta\),
so
\[
(a+r)^2>r^2+d^2\alpha^4+2d^2\alpha^3\beta
= r^2+d^2\alpha^3(\alpha+2\beta)=r^2+M=(a+r)^2,
\]
a contradiction. Therefore
\(
a<d\alpha^2\).

The argument for \(b\) is identical,
so
\( b<d\beta^2 \).

Since \(a,b\in\Z\), we have
\(
a\leqslant d\alpha^2-1\) and \(b\leqslant d\beta^2-1\).
Therefore
\[
r^2=ab+1
\leqslant (d\alpha^2-1)(d\beta^2-1)+1
= d^2\alpha^2\beta^2-d\alpha^2-d\beta^2+2.
\]
Now \(d\alpha^2\geqslant 1\) and \(d\beta^2\geqslant 2\) (because \(d\geqslant 1\), \(\alpha\geqslant 1\),
\(\beta>\alpha\)), so
\[
d\alpha^2+d\beta^2>2, \quad
\text{ and hence } 
r^2<d^2\alpha^2\beta^2,
\]
a contradiction. Therefore, no such regular integral Diophantine triple exists.
\end{proof}

\begin{corollary}
If \(\{a,b,c\}\) is a regular integral Diophantine triple, then the induced elliptic curve
\[
E:y^2=(ax+1)(bx+1)(cx+1)
\]
has torsion subgroup \( E(\Q)_{\text{tors}} \cong \Z/2\Z\times \Z/2\Z\).
\end{corollary}
\begin{proof}
    For elliptic curves induced by Diophantine triples in positive integers $\{a, b, c\}$, results of Dujella \cite[Theorem 2]{DujeJTNB} (see also \cite{DujeiMiljen}) show that the only possible torsion subgroups are $\Z/2\Z \times \Z/2\Z$ and $\Z/2\Z \times \Z/6\Z$. Our \cref{thm:no-regular-2x6} excludes \(\twotimessix\).
\end{proof}

\section{Torsion growth over quadratic fields}\label{sec:quadratic-growth}

Let
\[
E_{a,b,c}:\ y^2=(x+ab)(x+ac)(x+bc),
\]
where
\[
c=a+b+2r,\quad r^2=ab+1,\quad r>0.
\]
We have proven that $\Z/2\Z \times \Z/6\Z $ cannot appear over $\Q$, i.~e.
\(
E_{a,b,c}(\Q)_{\tors}\cong \Z/2\Z\times \Z/2\Z.
\)

We are now interested in whether $\twotimessix$ torsion group can appear over a quadratic field, for an elliptic curve induced by a regular Diophantine triple in (rational) integers. We use the following theorem by Najman, on how torsion can grow for any $E/\Q$.

\begin{theorem}[{\protect{\cite[Theorem 2]{najman_ratcubic}}}]\label{thm:PhiQ2}
Let \(E/\mathbb{Q}\) be an elliptic curve and let \(K\) be a quadratic field.
Then
\[
E(K)_{\tors}\cong
\begin{cases}
\mathbb{Z}/m\mathbb{Z}, & m=1,\dots,10,12,15,16,\\
\mathbb{Z}/2\mathbb{Z}\oplus \mathbb{Z}/2m\mathbb{Z}, & m=1,\dots,6,\\
\mathbb{Z}/3\mathbb{Z}\oplus \mathbb{Z}/3m\mathbb{Z}, & m=1,2,\\
\mathbb{Z}/4\mathbb{Z}\oplus \mathbb{Z}/4\mathbb{Z}. &
\end{cases}
\]
\end{theorem}

\noindent By this classification, if \(K/\Q\) is quadratic, then the only possibilities for
\(E_{a,b,c}(K)_{\tors}\) are the possible growths from \(\Z/2\Z\times \Z/2\Z\), which are
\[
\mathbb{Z}/2\mathbb{Z} \times \mathbb{Z}/2n\mathbb{Z}
\quad (n=1,2,\dots,6),
\quad
\mathbb{Z}/4\mathbb{Z} \times \mathbb{Z}/4\mathbb{Z}.
\]

Our \href{https://github.com/NikolaAdzaga/TorsionRegularTriples/blob/main/over_quadratic_fields.m}{computations} give examples of growth to
\(\Z/2\Z\times \Z/4\Z\):
\[
E_{1,3,8}\left(\Q(\sqrt{-5})\right)_{\tors}
 \cong \Z/2\Z\times \Z/4\Z,
\quad \text{and} \quad 
E_{7,9,32}\left(\Q(i)\right)_{\tors}
 \cong \Z/2\Z\times \Z/4\Z.
\]
On the other hand, in our search we found no example with quadratic
\(3\)-torsion, i.~e. no example with torsion
\(\Z/2\Z\times \Z/6\Z\) or \(\Z/2\Z\times \Z/12\Z\).
This motivates the following discussion.

\medskip

For an elliptic curve \(E\) and point \((x, y)\) of odd order \(n\), its \(x\)-coordinate is a root of the \(n\)-division polynomial of \(E\) (see, e.~g.~\cite[§13.9]{Husemoller2004}). For
\( E:\ y^2=x^3+Ax^2+Bx+C \)
the \(3\)-division polynomial is
\(
\psi_3(x)=3x^4+4Ax^3+6Bx^2+12Cx+(4AC-B^2).
\)
For the curve \(E_{a,b,c}\) we have
\( A=ab+ac+bc \), \( B=abc(a+b+c) \), \( C=(abc)^2 \),
hence
\begin{align*}
\psi_3(x)
=
& 3x^4
+4(ab+ac+bc)x^3
+6abc(a+b+c)x^2+\\
&+12a^2b^2c^2x
+a^2b^2c^2\left(4(ab+ac+bc)-(a+b+c)^2\right).
\end{align*}
Since \(c=a+b+2r\) and \(r^2=ab+1\), one checks that
\[
4(ab+ac+bc)-(a+b+c)^2=-4,
\]
so in our family \(E_{a,b,c}\), the \(3\)-division polynomial is
\begin{equation}\label{eq:psi3-family}
\psi_3(x)
=
3x^4
+4(ab+ac+bc)x^3
+6abc(a+b+c)x^2
+12a^2b^2c^2x
-4a^2b^2c^2.
\end{equation}

\begin{lemma}\label{lem:quadratic-3torsion-psi3}
Let \(F\) be a field of characteristic different from \(2\) and \(3\), and let
\(E/F\) be an elliptic curve given by \( E:\ y^2=f(x)\),
where \(f\in F[x]\) is a separable cubic. Then \(E\) has a point of order \(3\)
over an extension of \(F\) of degree at most \(2\) if and only if the
\(3\)-division polynomial \(\psi_3(x)\) has a root in \(F\).
\end{lemma}

\begin{proof}
If \(x_0\in F\) is a root of \(\psi_3(x)\), then a corresponding point of order
\(3\) has the form
\( P=(x_0,\sqrt{f(x_0)}) \).
Since \(f(x_0)\in F\), this point is defined over \(F\) or over a quadratic extension of \(F\).

Conversely, suppose that \(P\in E(L)\) has order \(3\), where \([L:F]\leqslant 2\).
If \(L=F\), then \(x(P)\in F\), so \(\psi_3(x)\) has a root in \(F\). Otherwise \(L/F\) is quadratic. Let \(\sigma\) be the nontrivial automorphism in \(\operatorname{Gal}(L/F)\).

If \(\sigma(P)\in\langle P\rangle\), then
\(\sigma(P)=\pm P\), and hence
\(
\sigma(x(P))=x(\sigma(P))=x(\pm P)=x(P)\). Thus \(x(P)\in F\).

If \(\sigma(P)\notin\langle P\rangle\), then
\(
R=P+\sigma(P)
\)
is a nonzero point of order \(3\), and
\(\sigma(R)=\sigma(P)+P=R\).
Thus \(R\in E(F)\), so \(x(R)\in F\). In either case, \(\psi_3(x)\) has a root in \(F\).
\end{proof}

We could now try to see when our polynomial \(\psi_3\) in \eqref{eq:psi3-family} has a rational root, but we first simplify it by writing our curve in Legendre form.

\begin{lemma}\label{lem:legendre-quad-3torsion}
Let \(F\) be a field of characteristic different from \(2\) and \(3\), let
\(\lambda\in F\setminus\{0,1\}\), and let \(d\in F^\times\). Consider the
quadratic twist of the Legendre curve
\(E_{d,\lambda}: Y^2=d\,u(u-1)(u-\lambda)\). Then \(E_{d,\lambda}\) has a
point of order \(3\) over an extension of \(F\) of degree at most \(2\) if and
only if there exists \(s\in F\), \(s^2\neq 1\), such that
\[
\lambda = -\frac{2s+1}{(s-1)(s+1)^3}
\quad\text{or}\quad
\lambda = \frac{2s-1}{(s-1)^3(s+1)}.
\]
\end{lemma}

\begin{proof}
The \(3\)-division polynomial of the Legendre curve
\(E_\lambda: y^2=u(u-1)(u-\lambda)\) is
\(\psi_3(u)=3u^4-4(1+\lambda)u^3+6\lambda u^2-\lambda^2\). The same polynomial
gives the \(u\)-coordinates of all the \(3\)-torsion points on every
quadratic twist \(E_{d,\lambda}\) (since quadratic twisting does not alter the first coordinate).

By Lemma \ref{lem:quadratic-3torsion-psi3}, \(E_{d,\lambda}\) has a point of
order \(3\) over an extension of degree at most \(2\) if and only if
there exists \(u\in F\) such that \(\psi_3(u)=0\). Viewing this as a
quadratic equation in \(\lambda\), we get
\(\lambda^2+(4u^3-6u^2)\lambda-(3u^4-4u^3)=0\). Its discriminant is
\(16u^3(u-1)^3\). Since \(\lambda\in F\), this discriminant must be a square in
\(F\). Therefore \(u(u-1)\) is a square in \(F\). Write \(u(u-1)=v^2\).

The conic \(v^2=u(u-1)\) is parametrized by the line \(v=su\) through
\((0,0)\), giving \(u=1/(1-s^2)\) and \(v=s/(1-s^2)\). Substituting this into
the quadratic formula for \(\lambda\) gives
\[
\lambda = -\frac{2s+1}{(s-1)(s+1)^3}
\quad\text{or}\quad
\lambda = \frac{2s-1}{(s-1)^3(s+1)}.
\]
The converse follows by reversing the calculation: either displayed value of
\(\lambda\), with \(u=1/(1-s^2)\), gives a root of \(\psi_3(u)\) in \(F\), and
then Lemma~\ref{lem:quadratic-3torsion-psi3} gives a point of order \(3\) over
\(F\) or over a quadratic extension of \(F\).
\end{proof}

We now apply Lemma~\ref{lem:legendre-quad-3torsion} to the curve
\(E_{a,b,c}\). The three ramification points of the \(x\)-coordinate map are
\(-ab\), \(-ac\), and \(-bc\). Put
\(u=(x+ab)/(ab-ac)\). Then these three points are sent to \(0\), \(1\), and
\begin{equation}
\label{eq:lambda-family}
\lambda
=
\frac{(-bc)-(-ab)}{(-ac)-(-ab)}
=
\frac{bc-ab}{ac-ab}
=
\frac{b(c-a)}{a(c-b)}
=
\frac{b(b+2r)}{a(a+2r)}.
\end{equation}
In the coordinate \(u\), the curve \(E_{a,b,c}\) becomes a quadratic twist of
the Legendre curve \(y^2=u(u-1)(u-\lambda)\).

\begin{proposition}\label{prop:lambda-parametrization}
Assume that an elliptic curve \(E_{a,b,c}\), induced by a regular Diophantine triple \(\{ a, b, c\}\),  acquires a point of order \(3\) over a quadratic
field. Then there exists \(s\in \Q\), \(s^2\neq 1\), such that
\[
\frac{b(b+2r)}{a(a+2r)}
=
-\frac{2s+1}{(s-1)(s+1)^3}
\quad\text{or}\quad
\frac{b(b+2r)}{a(a+2r)}
=
\frac{2s-1}{(s-1)^3(s+1)}.
\]
\end{proposition}

\begin{proof}
By the computation above, \(E_{a,b,c}\) is a quadratic twist of the Legendre
curve with parameter \(\lambda=b(b+2r)/(a(a+2r))\). The claim follows directly
from Lemma~\ref{lem:legendre-quad-3torsion}.
\end{proof}

Proposition~\ref{prop:lambda-parametrization} reduces the existence of
quadratic \(3\)-torsion in this family to the rational Diophantine problem
\[
\frac{b(b+2r)}{a(a+2r)}
=
-\frac{2s+1}{(s-1)(s+1)^3}
\quad\text{or}\quad
\frac{b(b+2r)}{a(a+2r)}
=
\frac{2s-1}{(s-1)^3(s+1)},
\]
together with the relation \(r^2=ab+1\). In particular, a proof that this
system has no rational solutions will rule out the torsion groups
\(\Z/2\Z\times \Z/6\Z\) and \(\Z/2\Z\times \Z/12\Z\) for the family
\(E_{a,b,c}\).

\begin{theorem}
Let \( E_k=E_{k-1,k+1,4k}\) be an elliptic curve induced by a triple $\{k-1, k+1, 4k\}$ where $k\in \Z\setminus\{0,\pm1\} $. Then \(E_k\) does not acquire a point of order \(3\) over any quadratic field. Consequently, for every quadratic field \(K/\Q\),
\[
E_k(K)_{\tors}
\not\cong \Z/2\Z\times\Z/6\Z,
\quad
E_k(K)_{\tors}
\not\cong \Z/2\Z\times\Z/12\Z.
\]
\end{theorem}
\begin{remark}
    We work with integral triples, but it is worth noting that our entire proof is valid for $k\in\Q$, provided one eliminates further degenerate $k\in\{ -\frac 13, \frac 13\}$, which do not give a triple to start with.
\end{remark}

\begin{proof}
For this family we have
\[
a=k-1,\quad b=k+1,\quad r=k,\quad c=4k,
\]
hence
\[
\lambda=\frac{b(b+2r)}{a(a+2r)}
=\frac{(k+1)(3k+1)}{(k-1)(3k-1)}.
\]
If \(E_k\) acquires a point of order \(3\) over a quadratic field, then by
Proposition~\ref{prop:lambda-parametrization} there exists \(p\in\Q\) such that
\[
\frac{(k+1)(3k+1)}{(k-1)(3k-1)}
=
-\frac{2p+1}{(p-1)(p+1)^3}
\]
or
\[
\frac{(k+1)(3k+1)}{(k-1)(3k-1)}
=
\frac{2p-1}{(p-1)^3(p+1)}.
\]
The second equation is obtained from the first by replacing \(p\) with \(-p\),
so it suffices to consider the first one.

After clearing denominators, the first equation becomes
\[
3p^3(p+2)k^2+4(p^4+2p^3-4p-2)k+p^3(p+2)=0.
\]
Thus its discriminant must be a square in \(\Q\). A calculation gives
\[
\Delta
=
4(p^2-2p-2)(p^6+6p^5+18p^4+16p^3-12p^2-24p-8).
\]
Hence \(p\) determines a rational point on the genus-\(3\) curve
\[
C':\ y^2=(p^2-2p-2)(p^6+6p^5+18p^4+16p^3-12p^2-24p-8),
\]

which, after $p=x-1$, gives 
\[
C:\ y^2= 1 - 4 x + 4 x^2 - 28 x^3 + 70 x^4 - 28 x^5 + 4 x^6 - 4 x^7 + x^8,
\]
where the hyperelliptic polynomial is self-reciprocal (i.~e.~palindromic). 

A Magma computation in the next section shows that the only (affine) rational \(p\)-coordinates on \(C'\) are
\[
p\in\left\{-2,-1,-\frac12,0,1\right\}.
\]
For \(p=\pm1\), the expression
\(
-(2p+1)/((p-1)(p+1)^3)
\)
is undefined. For \(p=0\) and \(p=-2\), it equals \(1\), while for
\(p=-\frac12\), it equals \(0\). Thus the only resulting values of \(\lambda\)
are \(0\) and \(1\), which are the singular Legendre values.

Therefore no nonsingular curve \(E_k\) with \(k\in\Z\setminus\{0,\pm1\}\) can
acquire a point of order \(3\) over a quadratic field.
\end{proof}

\section{Rational points on the hyperelliptic genus \texorpdfstring{$3$}{3} curve}

We have reduced our problem to a computation of all rational points on
\[
C:\ y^2= 1 - 4 x + 4 x^2 - 28 x^3 + 70 x^4 - 28 x^5 + 4 x^6 - 4 x^7 + x^8.
\]
The actual determination of $C(\Q)$ is the most algorithmic part of the paper. 

We give an exposition of the Chabauty--Coleman method with Mordell--Weil sieve, in hopes of encouraging researchers in the field of Diophantine \(m\)-tuples to apply these methods.

\begin{theorem}\label{thm:points}
Let \(C\) be the smooth projective model of the affine curve
\[
C:\ y^2= 1 - 4 x + 4 x^2 - 28 x^3 + 70 x^4 - 28 x^5 + 4 x^6 - 4 x^7 + x^8.
\]
Then \(C\) has exactly the following rational points
\[
C(\Q)=
\left\{
(-1,\pm 12),\,
(0,\pm 1),\,
\left(\frac12,\pm \frac{3}{16}\right),\,
(1,\pm 4),\,
(2,\pm 3),\,
\infty^+,\infty^-
\right\}.
\]
\end{theorem}
Since the defining polynomial has even degree and leading coefficient \(1\), the smooth projective model has two rational points at infinity, denoted \(\infty^+\) and \(\infty^-\). Equivalently, after putting \(t=1/x\) and \(Y=y/x^4\), the equation near infinity becomes
\( Y^2 = 1 - 4t + 4t^2 - 28t^3 + \cdots + t^8 \) (since our hyperelliptic polynomial is self-reciprocal, i.~e.~palindromic, we get the same polynomial in \(t=1/x\)),
so at \(t=0\) one has \(Y=\pm 1\). These two points are denoted by \(\infty^+\) and \(\infty^-\).

All computations below were carried out in {\sc Magma} \cite{BosmaCannonPlayoust1997} in the file \href{https://github.com/NikolaAdzaga/TorsionRegularTriples/blob/main/palindromic_curve.m}{palindromic\textunderscore curve.m};
the total running time is slightly below one minute.
Since $\operatorname{disc}(f)=2^{40}\cdot 3^{12}$, the curve has good reduction at
every prime $q\geqslant 5$, so all primes used below are primes of good reduction.

\subsection{The Mordell--Weil group}

Let $J$ be the Jacobian of $C$. A $2$-descent
(\cite{Stoll2001}, as implemented in {\sc Magma}) gives
$\operatorname{rank} J(\Q)\leqslant 2$. With base point $P_0=\infty^-$, we embed the curve into its Jacobian \(C(\Q) \hookrightarrow J(\Q), P \mapsto [P-P_0]\). The classes
\[
  D_1 = [(-1,-12)-\infty^-], \quad D_2 = [(0,-1)-\infty^-]
\]
have regulator $\det\left(\langle D_i,D_j\rangle\right)\approx 0.13175\neq 0$
with respect to the canonical height pairing, so they are independent and
$\operatorname{rank} J(\Q)=2$. Throughout, we work with the finite-index subgroup
$\Gamma=\langle D_1,D_2\rangle\cong\Z^2$ of $J(\Q)$.

Torsion will play no role: $J(\Q)_{\tors}$ injects into $J(\F_q)$ for
every odd prime $q$ of good reduction, and
\[
  \gcd\left(\#J(\F_{11}),\,\#J(\F_{13}),\,\#J(\F_{29})\right)
  =\gcd\left(2^8\cdot 7,\; 2^8\cdot 13,\; 2^4\cdot 3\cdot 5\cdot 97\right)=16,
\]
so $\#J(\Q)_{\tors}$ divides $16$. All the moduli we use below are odd,
so torsion does not affect the computations.

\subsection{Chabauty--Coleman method}

Since $\operatorname{rank} J(\Q)=2 < 3 = g$, the method of Chabauty and Coleman
applies \cite{Chabauty1941,Coleman1985}; see \cite{SiksekOhrid2014, McCallumPoonen2012} for an introduction. There is a nonzero regular differential $\omega$ on $C$ over
$\Q_{19}$ whose Coleman integral $\eta(P)=\int_{P_0}^{P}\omega$ vanishes on $J(\Q)$; in particular $\eta$ vanishes at every
$P\in C(\Q)$. We computed $\omega$ and the zeros of $\eta$ with the
Coleman integration package of Balakrishnan and Tuitman
\cite{BalakrishnanTuitman2020,BalakrishnanTuitmanCode} at $p=19$.

The curve has $\#C(\F_{19})=24$, so $C(\Q_{19})$ splits into $24$ residue
disks, and $\eta$ turns out to have exactly one zero in each of them.
Twelve of these zeros are the twelve known rational points of
Theorem~\ref{thm:points}; since each of their disks contains no further zero of
$\eta$, those disks contain no further rational points. The remaining twelve
zeros are points lying in the disks above
\[
  (x,y)\equiv (3,\pm 2),\ (7,\pm 5),\ (8,\pm 1),\ (11,\pm 2),\ (12,\pm 7),\
  (13,\pm 8) \pmod{19},
\]
and we have to show that none of these twelve ("extra") disks contains a
rational point. This is what the Mordell--Weil sieve does.

\subsection{The Mordell--Weil sieve}

The idea of the sieve \cite{Scharaschkin1999,Flynn2004,BruinStoll2010} is that
a hypothetical rational point would give compatible images modulo every
prime, and one shows that no global point can satisfy all the resulting local conditions. Although we give quite a few details and explanations, we also recommend Siksek's lecture notes \cite{SiksekOhrid2014} as they contain a genus-\(2\) example for both the Coleman integration and the Mordell--Weil sieve.

Suppose $R\in C(\Q)$ lies in one of the twelve extra disks, and consider
$x=[R-P_0]\in J(\Q)$. As \(J(\Q)\) is infinite, to perform some computations, we pass to a finite quotient (by looking at its reduction modulo \(m\)). Fix the modulus
\(  m = 273 = 3\cdot 7\cdot 13 \),
precisely the odd part of $\#J(\F_{19})=2^5\cdot 3\cdot 7\cdot 13$.
Let \( J(\Q)_{\mathrm{free}} = J(\Q)/J(\Q)_{\tors} \cong \Z^2 \) denote the free part of the Mordell--Weil group. The inclusion \(\Gamma \hookrightarrow J(\Q)\) gives a map \(\Gamma/m\Gamma \to J(\Q)_{\mathrm{free}}/mJ(\Q)_{\mathrm{free}}\). We will be computing in \( \Gamma/m\Gamma\), but because the index of $\Gamma$ in \(J(\Q)\) is coprime to \(m\) (for proof, see the next subsection), we have 
\[
  J(\Q)_{\mathrm{free}}/mJ(\Q)_{\mathrm{free}} \cong \Gamma/m\Gamma \cong (\Z/m\Z)^2
\]
so every rational point \(R\), i.~e.~every $x$, determines a pair $(a,b)\in(\Z/273\Z)^2$ with
$x\equiv aD_1+bD_2$. The choice of an odd modulus also avoids the rational torsion, whose order divides \(16\). The sieve collects necessary conditions on $(a,b)$:

\begin{enumerate}
\item[(C1)] \emph{At $p=19$.} Since \( \#J(\F_{19}) = 32\cdot 273\), multiplication by $32$ maps
$J(\F_{19})$ onto a cyclic group of order $273$, and reducing
$x\equiv aD_1+bD_2$ modulo $19$ gives
\[
  32\left(a\bar D_1+b\bar D_2\right) \;=\; 32\,[\bar R-\bar P_0]
  \quad\text{in } J(\F_{19}),
\]
where $\bar R\in C(\F_{19})$ is the known reduction determined by the disk.
This single condition cuts the $273^2=74\,529$ pairs $(a,b)$ down to $273$
per disk.
\item[(C2)] \emph{At an auxiliary prime $q$.} Whatever $R$ is, its reduction lies on
$C(\F_q)$. Setting $d_q=\gcd(m,\#J(\F_q))$ and $n_q=\#J(\F_q)/d_q$, this forces
\[
  n_q\left(a\bar D_1+b\bar D_2\right)\;\in\;
  n_q\cdot\left\{\,[\bar P-\bar P_0] : \bar P\in C(\F_q)\,\right\}
  \quad\text{in } J(\F_q).
\]
The right-hand side has at most $\#C(\F_q)\approx q$ elements inside a group of
exponent $d_q$, so this is a strong condition whenever $d_q$ is large. This is also frequently a guide in choosing which primes to use as auxiliary primes for the sieve.
\end{enumerate}

Both types of conditions are necessary, so if for some disk
no pair $(a,b)$ satisfies all of them simultaneously, that disk contains
no rational point.

The auxiliary primes were found by a search over small primes of good
reduction, keeping only those with $d_q>1$ that actually eliminated surviving
pairs; six primes sufficed. Table~\ref{tab:sieve} shows the run. The prime
$q=113$ is the decisive one: it is the only auxiliary prime in our list such that $273\mid\#J(\F_{q})$.

\begin{table}[ht]
\centering
\begin{tabular}{rllr}
\toprule
$q$ & $\#J(\F_q)$ & $d_q=\gcd(273,\#J(\F_q))$ & surviving pairs \\
\midrule
--  &             &      & $12\cdot 273=3276$ \\
11  & $2^8\cdot 7$              & $7$   & $2340$ \\
13  & $2^8\cdot 13$             & $13$  & $1980$ \\
23  & $2^6\cdot 3^3\cdot 7$     & $21$  & $880$  \\
37  & $2^7\cdot 13\cdot 47$     & $13$  & $660$  \\
43  & $2^8\cdot 3^3\cdot 13$    & $39$  & $520$  \\
113 & $2^4\cdot 3^2\cdot 7\cdot 13\cdot 113$ & $273$ & $0$ \\
\bottomrule
\end{tabular}
\caption{The Mordell--Weil sieve: total number of surviving pairs
$(a,b)\in(\Z/273\Z)^2$, summed over the twelve extra residue disks, after
each auxiliary prime.}
\label{tab:sieve}
\end{table}

We conclude that none of the twelve extra disks contains a
rational point, which proves Theorem~\ref{thm:points}.

\subsection{Saturation}

The sieve used that every $x\in J(\Q)$ is, modulo $mJ(\Q)$ and torsion, an
integral combination of $D_1$ and $D_2$. This requires the index
$n=[J(\Q)_{\mathrm{free}}:\Gamma]$ to be coprime to $m=273$. This means that we cannot divide by any prime \(\ell\mid m\) in our \(\Gamma\). More precisely, for a prime $\ell$, the subgroup $\Gamma$ is called
\emph{$\ell$-saturated} in $J(\Q)$ if there is no $y\in J(\Q)$ with
$\ell y\in\Gamma+J(\Q)_{\tors}$ but
$y\notin\Gamma+J(\Q)_{\tors}$; equivalently, $\ell\nmid n$. So we must
check that $\Gamma$ is $\ell$-saturated for $\ell\in\{3,7,13\}$.

This is done by reduction. Suppose some $x=aD_1+bD_2$ with
$(a,b)\not\equiv(0,0)\pmod\ell$ were $\ell$-divisible in $J(\Q)$ modulo
torsion, say $x=\ell y+t$ with $y\in J(\Q)$ and
$t\in J(\Q)_{\tors}$. Reducing modulo an odd prime $q$ of good
reduction and multiplying by $\#J(\F_q)/\ell$ gives
\[
  \frac{\#J(\F_q)}{\ell}\,\bar x
  \;=\;\#J(\F_q)\,\bar y+\frac{\#J(\F_q)}{\ell}\,\bar t\;=\;0
  \quad\text{in } J(\F_q),
\]
where both summands vanish: the second because $t$ has $2$-power order while $\ell$ is odd, so $\#J(\F_q)/\ell$ retains the
full $2$-part of $\#J(\F_q)$. Hence a single prime $q$ with
$(\#J(\F_q)/\ell)\,(a\bar D_1+b\bar D_2)\neq 0$ is sufficient to conclude that
$aD_1+bD_2$ is \emph{not} $\ell$-divisible. If such a prime is found for
every nonzero class $(a,b)\in(\Z/\ell\Z)^2$, then $\Gamma$ is $\ell$-saturated: since $\ell\mid n$ would force some nonzero class of $\Gamma/\ell\Gamma$ to be
$\ell$-divisible in $J(\Q)_{\mathrm{free}}$.
We restrict to primes $q$ with $\ell\,\|\,\#J(\F_q)$: for such $q$,
multiplication by $\#J(\F_q)/\ell$ is the projection onto
the $\ell$-part of $J(\F_q)$ (i.~e.~onto \( \Z/\ell\Z\)), so the test detects precisely whether
$\bar x\in\ell J(\F_q)$.
In our situation, \(\ell\)-saturation was quickly confirmed:
$\ell=3$ by $q\in\{19,127\}$, $\ell=7$ by $q\in\{11,17\}$, and
$\ell=13$ by $q\in\{13,19\}$.

\medskip
The functions used for sieving and saturation are stored in a separate file \href{https://github.com/NikolaAdzaga/TorsionRegularTriples/blob/main/MWS_and_saturation.m}{MWS\textunderscore and\textunderscore saturation.m} as we will need them for one more computation.

\section{About a family of induced elliptic curves}
For an integer \(k\geqslant 1\), the set \(\{1,2k^2,2k^2+2k+1\}\) is a
\(D(-k^2)\)-triple. From now on, \(E_k\) denotes the elliptic curve induced
by this triple:
\[ E=E_k \, : \, y^2=(x-k^2)(2k^2x-k^2)((2k^2+2k+1)x-k^2), \]
and again, an extension of the starting triple gives an integral point on $E=E_k$. For this triple, in \cite{AF}, we have previously shown that, if $k\equiv 2, 3 \pmod{6}$, then there is no extension. In general, we conjectured that it can only be extended by $d=8k^2+4k+1$ (and in that case $7k^2+4k+1$ must be a square).

For determining the torsion, we will need the following well-known halving criterion.

\begin{theorem}[{\cite[Ch.~1, Theorem 4.1]{Husemoller2004}}]\label{thm:halvingCrit}
Let \(E\) be an elliptic curve over a field \(F\), given by
\(E:\ y^2=(x-\alpha)(x-\beta)(x-\gamma)\), where
\(\alpha,\beta,\gamma\in F\). For a point \((x',y')\in E(F)\), there exists
\((x,y)\in E(F)\) such that \(2(x,y)=(x',y')\) if and only if
\(x'-\alpha\), \(x'-\beta\), and \(x'-\gamma\) are squares in \(F\).
\end{theorem}

To apply it, using the following straightforward lemma, we find a \(\Q\)-isomorphic model of $E_k$.

\begin{lemma}\label{lem:Qiso}
Let \(K\) be a field of characteristic different from \(2\), and let
\(a,c\in K^\times\) and \(b,d\in K\). The curve
\(y^2=z(az+b)(cz+d)\) is \(K\)-isomorphic to
\(Y^2=X(X+bc)(X+ad)\), via \(X=acz\) and \(Y=acy\).
\end{lemma}

\begin{proof}
This follows by substituting \(z=X/(ac)\) and \(y=Y/(ac)\). Indeed,
\(Y^2=a^2c^2z(az+b)(cz+d)=X(X+bc)(X+ad)\).
\end{proof}

\begin{proposition}\label{prop:no4}
Let \(k\geqslant 1\), and let
\[
E_k:\ y^2=(x-k^2)(2k^2x-k^2)((2k^2+2k+1)x-k^2).
\]
Then \(E_k(\Q)_{\tors}\) contains neither
\(\mathbb Z/2\mathbb Z\times \mathbb Z/4\mathbb Z\) nor
\(\mathbb Z/2\mathbb Z\times \mathbb Z/8\mathbb Z\).
\end{proposition}

\begin{proof}
Put \(c=2k^2+2k+1\). Translating the root \(x=k^2\) to the origin to obtain $y^2=x(2k^2(x+k^2)-k^2)(c(x+k^2)-k^2)$, and then, by Lemma \ref{lem:Qiso},
putting
\[
X=2k^2c(x-k^2),\quad Y=2k^2c\,y,
\]
we obtain a curve \(\Q\)-isomorphic to \(E_k\), namely
\[
E_k':\quad Y^2=X(X+k^2c(2k^2-1))(X+4k^5(k+1)).
\]
Thus it is enough to prove the claim for \(E_k'\).

The curve \(E_k'\) has full rational \(2\)-torsion. By the halving criterion (\cref{thm:halvingCrit}) for
full rational \(2\)-torsion, a point \((e_i,0)\) is divisible by \(2\) in
\(E_k'(\Q)\) only if both differences \(e_i-e_j\) and \(e_i-e_\ell\)
are rational squares.

The three roots of the right-hand side of \(E_k'\) are
\( 0\), \( -k^2c(2k^2-1)\),\( -4k^5(k+1)\).
Since \(k\geqslant 1\), the two nonzero roots are negative. Hence neither of the
two corresponding \(2\)-torsion points can be divisible by \(2\), because its
difference with the root \(0\) is negative.

It remains to consider the point \((0,0)\). If \((0,0)\) were divisible by
\(2\), then in particular
\( 0-(-4k^5(k+1))=4k^5(k+1) \) would have to be a rational square. This is impossible, since
\( 4k^5(k+1)=4k^4\cdot k(k+1)\), and \(k(k+1)\) is not a square for \(k\geqslant 1\), since  \( k^2<k(k+1)<(k+1)^2\).
Thus no rational \(2\)-torsion point is divisible by \(2\), so
\(E_k(\Q)_{\tors}\) cannot contain
\(\mathbb Z/2\mathbb Z\times\mathbb Z/4\mathbb Z\). Consequently it also
cannot contain \(\mathbb Z/2\mathbb Z\times\mathbb Z/8\mathbb Z\).
\end{proof}

\begin{corollary}
    For the induced elliptic curve $E_k$ (and $k\geqslant 1$), the torsion can only be $E_k(\Q)_\tors \cong \Z / 2\Z \times \Z / 2\Z$ or $\Z / 2\Z \times \Z / 6\Z$.
\end{corollary}
\begin{proof}
    By Mazur's famous classification \cite{mazurtorzija}, since the full $2$-torsion is rational, the claim follows.
\end{proof}

We next reduce the possible occurrence of quadratic \(3\)-torsion to the
determination of rational points on a curve of genus \(3\).

\begin{lemma}\label{lem:Dminusksq-3torsion-genus3}
For an integer \(k\geqslant 1\), the curve \(E_k\) acquires a
point of order \(3\) over a field of degree at most \(2\) over \(\Q\) if and
only if there exists \(s\in\Q\), \(s^2\neq 1\), such that \((k,s)\) lies on 
\[ C_3:\, 4k^3(k+1)s^3(s+2)=(2k+1)(2s+1). \]
\end{lemma}

\begin{proof}
Put \(c=2k^2+2k+1\). As in the proof above, after translating the root
\(x=k^2\) to the origin and applying Lemma~\ref{lem:Qiso}, the curve \(E_k\)
is \(\Q\)-isomorphic to
\(E_k': Y^2=X(X+k^2c(2k^2-1))(X+4k^5(k+1))\).

The Legendre coordinate \(u=-X/(k^2c(2k^2-1))\) sends the three roots of
\(E_k'\) to \(0\), \(1\), and
\(\lambda_k=4k^3(k+1)/((2k^2-1)(2k^2+2k+1))\). Hence \(E_k'\), and therefore
also \(E_k\), is a quadratic twist of the Legendre curve
\(y^2=u(u-1)(u-\lambda_k)\).

By Lemma~\ref{lem:legendre-quad-3torsion}, the curve \(E_k\) acquires a point
of order \(3\) over an extension of degree at most \(2\) if and only if
\(\lambda_k\) is of one of the two forms appearing there. The two forms are
interchanged by replacing \(s\) by \(-s\), so it is enough to consider
\(\lambda_k=-(2s+1)/((s-1)(s+1)^3)\). Substituting the value of \(\lambda_k\)
and using
\((2k^2-1)(2k^2+2k+1)=4k^3(k+1)-(2k+1)\), we obtain, after simplification,
\(4k^3(k+1)s^3(s+2)=(2k+1)(2s+1)\). This gives the claimed equation for \(C_3\).
Conversely, reversing the calculation gives one of the two parametrizations
from Lemma~\ref{lem:legendre-quad-3torsion}, and hence gives a point of order
\(3\) over \(\Q\) or over a quadratic field.
\end{proof}

In the following proposition, we show that the only (affine) rational \(k\)-coordinates on \(C_3\) are \[ k\in \left\{ 0, -\frac 12, -\frac 23, -\frac 34, -1\right\}.\]

\begin{proposition}\label{prop:C3Q}
    Let \(C_3\) be the smooth projective model of the curve
    \[ C_3:\, 4k^3(k+1)s^3(s+2)=(2k+1)(2s+1). \]
    Then \(C_3\) has exactly eight rational points, the six affine points \( (k,s)\) and two at infinity:
\[ C_3(\Q) = \left\{
\left(0,-\frac12\right),\,
\left(-\frac12,0\right),\,
\left(-\frac12,-2\right),\,
\left(-\frac23,-\frac32\right),\,
\left(-\frac34,-\frac23\right),\,
\left(-1,-\frac12\right),\,
\infty_k,\,
\infty_s
\right\}.\]
\end{proposition}

\begin{proof}
Putting \(z=ks\), the defining equation becomes quadratic in \(k\), so taking its discriminant gives a hyperelliptic model
\( w^2 = 16z^8 - 64z^7 + 64z^6 - 32z^5 + 24z^4 - 16z^3 + 16z^2 - 8z + 1\).

    We apply the Chabauty--Coleman method with Mordell--Weil sieve, completely analogously to the proof of \cref{thm:points}. A {\sc Magma} computation shows that \(C_3\) has genus \(3\), and a point search finds the listed points. The Jacobian \(J\) has rank at most \(2\),
and two divisor classes supported on these points have nonzero regulator, so they generate a finite-index subgroup \(\Gamma \subseteq J(\Q)\).

For choosing the primes to work with, we computed the structure of the finite abelian groups \(J(\F_p)\) of our curve (for good primes \(p \leqslant 100\)). This shows that the torsion order divides \(16\). Moreover, since
\( 33\mid \#J(\F_p) \) for primes \(p=7, 31, 37, 73\), for the Chabauty--Coleman method we chose the prime \(p=7\). The computation leaves four residue disks not containing any known rational points. Taking \(m=33\), the odd part of \(\#J(\F_7)=528\), we verify \(\ell\)-saturation of \(\Gamma\) for \( \ell = 3, 11\). For auxiliary sieving primes, we chose \(q=31, 37, 73\), but then only \(31\) and \(37\) turned out to be sufficient. The code is available at \href{https://github.com/NikolaAdzaga/TorsionRegularTriples/blob/main/CC_MWS_final_curve.m}{CC\textunderscore MWS\textunderscore final\textunderscore curve.m}.
\end{proof}

    \begin{corollary}
        An elliptic curve \(E_k\) induced by a \(D(-k^2)\)-triple \( \{1, 2k^2, 2k^2+2k+1\} \) has torsion subgroup \( E_k(\Q)_\tors \cong \Z/2\Z \times \Z/2\Z\) for all integers \(k\geqslant 1\).
     Even for all rational \(k\), \(E_k(\Q)_\tors \not\cong \twotimessix\). Moreover, for \(k\in\Q \backslash \{0, -1/2, -1\}\), if \(E_k(K)_\tors \cong \twotimessix\) for a quadratic field \(K\), then \(k\in \{-2/3, -3/4\}\) and \(K=\Q(i)\).
    \end{corollary}

    \begin{proof}
        \(E_k\) has full rational \(2\)-torsion. After discarding the values \(k=0, -1/2, -1\) (for which our starting triple degenerates to a pair), Proposition \ref{prop:C3Q} shows that the only rational \(k\) for which \(3\)-torsion can occur over a field of degree at most \(2\) are \(k=-2/3, -3/4\). Direct computation (at the end of the same \href{https://github.com/NikolaAdzaga/TorsionRegularTriples/blob/main/CC_MWS_final_curve.m}{source file} as above) shows that in both cases the field is \(\Q(i)\), with torsion \(\twotimessix\), while over \(\Q\) the torsion remains \(\Z/2\Z \times \Z/2\Z\).
        \end{proof}

We next study the generic rank of this family.
Let \(K=\Q(k)\), and consider \(E\) as an elliptic curve over \(K\).
Putting
\(x=\frac{X}{2k^2(2k^2+2k+1)}+k^2\) and
\(y=\frac{Y}{2k^2(2k^2+2k+1)}\), the curve becomes
\(E':\ Y^2=X^3+A(k)X^2+B(k)X\), where
\(A(k)=8k^6+8k^5-2k^3-k^2\) and
\(B(k)=4k^7(k+1)(2k^2-1)(2k^2+2k+1)\). Moreover,
\(A(k)^2-4B(k)=k^4(2k+1)^2\).

To determine its generic rank, we use the following specialization criterion of Gusić and Tadić \cite{GusicTadic} (see \cite[Theorem 3.1]{GusicTadicGlasnik} for the original formulation).

\begin{theorem}[{\cite[Theorem 1.1]{GusicTadic}}]\label{thm:GusicTadic}
Let \(E\) be a nonconstant elliptic curve over \(\mathbb{Q}(t)\), given by the equation
\[
E=E(t):\quad y^2=(x-e_1)(x-e_2)(x-e_3),
\quad e_1,e_2,e_3\in \mathbb{Z}[t].
\]
Assume that \(t_0\in\mathbb{Q}\) satisfies the following condition.

\begin{enumerate}
\item[\textnormal{(GT)}] For every nonconstant square-free divisor \(h\in\mathbb{Z}[t]\) of
\[
(e_1-e_2)(e_1-e_3),\quad
(e_2-e_1)(e_2-e_3),\quad
\text{or}\quad
(e_3-e_1)(e_3-e_2),
\]
the rational number \(h(t_0)\) is not a square in \(\mathbb{Q}\).
\end{enumerate}

\noindent Then the specialization homomorphism
\( \sigma:E(\mathbb{Q}(t))\longrightarrow E(t_0)(\mathbb{Q}) \)
is injective.
\end{theorem}

\begin{theorem}
The elliptic curve over \(\Q(k)\) induced by the
\(D(-k^2)\)-triple \(\{1,2k^2,2k^2+2k+1\}\) has rank \(0\).
Equivalently, the induced elliptic surface has generic Mordell--Weil rank \(0\).
\end{theorem}
\begin{proof}
After translating the root \(x=k^2\) to the origin and using the isomorphism from
Lemma \ref{lem:Qiso}, the curve is \(\Q(k)\)-isomorphic to
\(Y^2=X(X^2+A(k)X+B(k))\), where
\(A(k)=8k^6+8k^5-2k^3-k^2\) and
\(B(k)=4k^7(k+1)(2k^2-1)(2k^2+2k+1)\).

More explicitly, this model factors as
\[
Y^2=X(X+k^2(2k^2-1)(2k^2+2k+1))(X+4k^5(k+1)).
\]
Thus the three roots are
\( 0,\, -k^2(2k^2-1)(2k^2+2k+1), \, -4k^5(k+1)\).

The difference of the two nonzero roots is
\[
4k^5(k+1)-k^2(2k^2-1)(2k^2+2k+1)=k^2(2k+1).
\]
Put
\( p=2k^2-1,\, q=2k^2+2k+1,\, r=2k+1\).
So to confirm the condition (GT) the nonconstant squarefree divisors which have to be checked are products
of factors from the following three lists:
\[
\{2,k,k+1,p,q\},\quad \{k,p,q,r\},\quad \{2,k,k+1,r\}.
\]

We specialize at \(k=14\). The corresponding values are
\( 2, 14, 15, 391, 421, 29 \).
Since
\( 14=2\cdot 7, 15=3\cdot 5, 391=17\cdot 23\),
and \(421\) and \(29\) are prime, no nonempty product involving at least one
nonconstant factor from any of the three lists above is a square in \(\Q\).
Changing the sign gives a negative value at \(k=14\), and hence also not a square.
Therefore the specialization map
\( E(\Q(k))\to E_{14}(\Q) \)
is injective.

For \(k=14\), the specialized curve is
\[
E_{14}:\ Y^2=X^3+64\,533\,196X^2+1\,041\,133\,338\,416\,640X.
\]
\href{https://github.com/NikolaAdzaga/TorsionRegularTriples/blob/main/generic_rank_0.m}{A direct computation} in {\sc Magma} gives
\(\operatorname{rank} E_{14}(\Q)=0\). By injectivity of specialization,
it follows that \(\operatorname{rank} E(\Q(k))=0\).
\end{proof}

Thus this family has no non-torsion section over \(\Q(k)\), and none of its torsion sections gives an extension of our original \(D(-k^2)\)-triple. Hence any extension can occur only on individual fibers, not generically. This can be seen as further support for our conjecture from \cite{AF}.

\section{Open questions and concluding remarks}
We have studied the torsion (and briefly, rank) of elliptic curves induced by regular Diophantine triples. We have shown that, for regular Diophantine triples in integers, the induced elliptic curve always has \(\Z/2\Z \times \Z/2\Z\) as rational torsion subgroup. Even for regular triples and the families studied in this manuscript, some questions remain open. Is there a regular triple in integers such that the induced elliptic curve has torsion \( \twotimessix\) over some quadratic field? More broadly, which torsion subgroups can appear over quadratic fields and do they appear infinitely often? For an elliptic curve \(E_k\), induced by a \(D(-k^2)\)-triple \(\{1, 2k^2, 2k^2+2k+1\}\), is \(\Z/2\Z \times \Z/2\Z\) the only possible torsion subgroup \(E_k(\Q)_\tors\) even for rational \(k\)? Many of these questions can be answered using the same methods, or more broadly, using methods from arithmetic geometry. Such methods have already proven to be useful even in the context of the classical integral extension problems: in \cite{AFFu}, it was necessary to find the integral points on a genus-\(2\) curve with Jacobian of rank \(2\), where the classical Chabauty method does not apply. A method by Gallegos-Ruiz \cite{GallegosRuiz2010,GallegosRuiz2019} to determine integral points on hyperelliptic curves was used to complement the more frequently used methods in this field (the hypergeometric method, linear forms in logarithms).

For the \(D(-k^2)\)-triple \(\{1, 2k^2, 2k^2+2k+1\}\) in positive integers, is our conjecture (from \cite[§4]{AF}) that it can only be extended by $d=8k^2+4k+1$ correct?
Is there any Diophantine triple in positive integers such that the induced elliptic curve has \(\twotimessix\)-torsion subgroup over \( \Q\)? Let us note that this last question also appears in Dujella's extensive list of open problems on Diophantine \(m\)-tuples and elliptic curves \cite{DujeOpenProblems2026}.

\section*{Declaration of generative AI in the manuscript preparation}
The proofs of \cref{thm:no-regular-2x6} and Lemma \ref{lem:BA-lt-betaalpha} were AI-assisted: this required work with ChatGPT and Gemini in many iterations; at the time (January-February 2026), neither was capable of generating the proofs independently. The original version included an additional assumption that \(c > 2b\), which was first relaxed to \(c > \frac 43 (b-a)\), and finally completely removed. In other sections, the Magma code for the Mordell-Weil sieve was written by ChatGPT (Plus) and edited by the author and Claude. All the code was written with an emphasis on clarity, rather than efficiency. The exposition in §5 was written with the help of Claude. The entire manuscript was checked by Claude.

\section*{Acknowledgements}
The author would like to thank Andrej Dujella for useful comments and for posing a number of interesting questions, some of which are, at least partially, answered in this manuscript. The author would also like to thank Goran Dražić for numerous comments that led to the improvement of the manuscript.

This study was conducted within the framework of the institutional project of the University of Zagreb Faculty of Civil Engineering – [GEM
– Graphs, elliptic and modular curves], under the programme line for enhancing scientific excellence, financed through
the National Recovery and Resilience Plan 2021–2026. The project is funded by the European Union – NextGenerationEU.

This work is supported by the Croatian Science Foundation under the project no.~IP-2022-10-5008, and also by the project, ``Implementation of cutting-edge research and its application as part of the Scientific Center of Excellence for Quantum and Complex Systems, and Representations of Lie Algebras'', Grant No.~PK.1.1.10.0004, co-financed by the European Union through the European Regional Development Fund - Competitiveness and Cohesion Programme 2021-2027.

\printbibliography
\end{document}